\documentclass[12pt]{amsart}
\usepackage[pdftex]{hyperref}
\hypersetup{
	colorlinks   = true, 
	urlcolor     = gray, %Color for external hyperlinks
	linkcolor    = red, %Color of internal links
	citecolor   =  blue%Color of citations
}
\usepackage{amsmath,amsfonts,amsthm,amssymb,graphicx,xcolor,enumitem}
\usepackage[all,2cell,ps]{xy}
\hypersetup{colorlinks=true, citecolor=[rgb]{0,0.5,1}}
\usepackage[margin=1.2 in, marginpar=1 in]{geometry}

\theoremstyle{plain}
\newtheorem{theorem}{Theorem}[section]

\newtheorem{lemma}[theorem]{Lemma}
\newtheorem{proposition}[theorem]{Proposition}
\newtheorem{conjecture}[theorem]{Conjecture}

\theoremstyle{definition}

\newtheorem{remark}[theorem]{Remark}

\newtheorem{question}[theorem]{Question}

\newcommand{\IC}{\mathbb{C}}

\newcommand{\IE}{\mathbb{E}}
\newcommand{\IF}{\mathbb{F}}

\newcommand{\IH}{\mathbb{H}}

\newcommand{\IL}{\mathbb{L}}

\newcommand{\IP}{\mathbb{P}}
\newcommand{\IQ}{\mathbb{Q}}
\newcommand{\IR}{\mathbb{R}}
\newcommand{\IS}{\mathbb{S}}

\newcommand{\IZ}{\mathbb{Z}}

\DeclareMathOperator{\End}{End}

\DeclareMathOperator{\Ric}{Ric}
\title[The signature of geometrically decomposable aspherical $4$-manifolds]{The signature of geometrically decomposable aspherical $4$-manifolds}

\author{}

\author{Luca F. Di Cerbo}
\email{ldicerbo@ufl.edu}
\address{Department of Mathematics, University of Florida, Gainesville, United States}
\author{Marco Golla}
\email{marco.golla@univ-nantes.fr}
\address{CNRS, Laboratoire Jean Leray, Nantes University, Nantes, France}

\date{}

\begin{document}
	
	\maketitle
	
	%%%%%%%%%%%%%%%%%%%%
	\begin{abstract}
		We construct examples of geometrically decomposable aspherical $4$-manifolds with non-zero signature. We show that all such $4$-manifolds satisfy the inequality (of Bogomolov--Miyaoka--Yau type)
		\[
		\chi\geq 3|\sigma|.
		\]
		We also construct examples attaining the equality that are non-geometric and have non-zero signature. Finally, we prove that for higher graph $4$-manifolds, with complex-hyperbolic vertices, the strict inequality always holds. Moreover, we construct infinitely many examples of higher graph $4$-manifolds with non-zero signature and prove that the inequality is strict and sharp in this class.
	\end{abstract}
		
	%%%%%%%%%%%%%%%%%%%%
	
%	\tableofcontents
		
	%%%%%%%%%%%%%%%%%%%%%%%%%%%%%%%%
	
\section{Introduction and statement of the main result}\label{intro}

%\MG{`of I. Singer' $\to$ `due to Singer' e `work of Atiyah' $\to$ `Atiyah's work'}
A long-standing conjecture due to Singer, inspired by Atiyah's influential work~\cite{Atiyah}, states that the $L^2$-Betti numbers of a closed aspherical manifold vanish except possibly in the middle-dimensional degree. In dimension four, this general statement implies the following conjecture about the geography of aspherical 4-manifolds.
%\MG{Ti ho tolto un `tantalizing'... te ne lascio uno per dopo}

\begin{conjecture}\label{MainC}
If $M$ is a closed, oriented, aspherical $4$-manifold, then
\[
\chi(M)\geq |\sigma(M)|
\] 
where $\chi(M)$ and $\sigma(M)$ are respectively the Euler number and signature of $M$.
\end{conjecture}

We refer to L\"uck's paper~\cite{Luck2} for more details concerning Conjecture~\ref{MainC}, and also for its connections with the Singer--Hopf problem and related ideas and problems discussed by Gromov~\cite[Section 8]{Gro93}.
%\MG{`of Gromov' $\to$ `discussed by Gromov'?}

%In this paper, we prove a strengthened and sharp version of Conjecture~\ref{MainC} for \textit{geometrically decomposable} aspherical $4$-manifolds. Recall that there are 19 geometries in dimension $4$~\cite{Filipkiewicz}. In the book \cite{Hil02}, Hillman extensively studied  the analogue of JSJ decompositions of $3$-manifolds in the context of $4$-manifolds. Contrarily to the case of 3-manifolds, not all $4$-manifolds decompose into geometric pieces. A geometrically decomposable $4$-manifolds is one that admits such a decomposition. The main result of this paper is the following.

In this paper, we prove a strengthened and sharp version of Conjecture~\ref{MainC} for \textit{geometrically decomposable} aspherical $4$-manifolds. Recall that the geometrization of low-dimensional manifolds was pioneered by W. Thurston; see, for example, his influential book ~\cite{Thurston}.  In dimension $4$~\cite{Filipkiewicz}, there are 19 geometries. In the book \cite{Hil02}, Hillman extensively studied the analog of JSJ decompositions of $3$-manifolds in the context of $4$-manifolds. Unlike 3-manifolds, not all 4-manifolds decompose into geometric pieces. A geometrically decomposable $ 4$-manifold is one that admits such a decomposition. The main result of this paper is the following.

\begin{theorem}\label{Tgeometric}
Let $M$ be a closed aspherical $4$-manifold admitting a geometric decomposition in the sense of Hillman. Then
	\begin{equation}\label{E3sigma}
	\chi(M) \ge 3|\sigma(M)|.
	\end{equation}
Moreover, there exist non-geometric aspherical $4$-manifolds that have non-vanishing signature and attain the equality.
\end{theorem}
%\MG{Per inciso, Hamenstadt ha cercato di dimostrare questa disuguaglianza per i fibrati in superfici iperboliche su superfici iperboliche. Forse dovremmo dire che non tutti questi fibrati hanno una decomposizione geometrica. (O almeno credo che non ce l'abbiano. Forse Hillman lo dice esplicitamente.)}

The class of geometrically decomposable $4$-manifolds includes two large classes of aspherical $4$-manifolds: \textit{extended graph $4$-manifolds} and \textit{higher graph $4$-manifolds}. These graph-like manifolds were introduced respectively by Frigerio, Lafont, and Sisto~\cite{LafontBook} and by Connell and Su\'arez-Serrato~\cite{CSS19} in arbitrary dimension. The common machinery behind their constructions is based on the idea of \textit{aspherical unions} of aspherical manifolds with aspherical and $\pi_1$-injective boundaries using a classical result of Whitehead~\cite{Whi39}.

Extended graph $n$-manifolds were defined by Frigerio, Lafont, and Sisto in~\cite[Definition 0.2]{LafontBook}. These manifolds are graph-like. They are decomposed into finitely many pieces (the vertices), each vertex is a manifold with torus boundary components (the edges), and the various pieces are glued together via affine diffeomorphisms. The interior of each vertex is diffeomorphic either to a finite-volume real-hyperbolic manifold with toral cusps (the \emph{pure pieces}), or to the product of a standard torus with a lower-dimensional finite-volume real-hyperbolic manifold with toral cusps (the \emph{product pieces}). In this paper, we also consider a larger class of extended graph $4$-manifolds. More precisely, we allow the vertices to be finite-volume real-hyperbolic $4$-manifolds, whose cusps are not necessarily toral.

\begin{theorem}\label{Treal}
	Let $Z$ be a closed, oriented, extended graph $4$-manifolds whose vertices are orientable finite-volume real-hyperbolic $4$-manifolds with possibly non-torus cusps. Then $\sigma(Z)=0$.
\end{theorem} 

Higher graph $n$-manifolds were defined by Connell and Su\'arez-Serrato in~\cite[Definition 1]{CSS19}. These manifolds are also graph-like. They are decomposed into finitely many pieces, each vertex is a manifold with nilmanifold boundary and the various pieces are glued together via affine diffeomorphisms. In particular, only diffeomorphic nilmanifolds can be paired in the gluing process. The interior of each vertex is diffeomorphic to a finite-volume complex-hyperbolic manifold with nilmanifold (or \emph{torus-like} in the terminology of~\cite{DG23}) cusps. In this paper, we consider a larger class of higher graph $4$-manifolds. Namely, we allow the vertices to be complex-hyperbolic surfaces with general infra-nilmanifold cusps. For such manifolds, we have a sharper inequality than in Theorem \ref{Tgeometric} and also the first examples of higher graph $4$-manifolds with non-vanishing signature.

\begin{theorem}\label{Tcomplex}
	For any closed, oriented, higher graph $4$-manifold $Z$ whose vertices are truncated complex-hyperbolic surfaces with cusps (not necessarily nilmanifold cusps), we have
	\begin{align}\label{TI}
	\chi(Z)>3|\sigma(Z)|.
	\end{align}
	Moreover, this inequality is sharp, i.e., there exists a sequence $\{Z_{n}\}$ of such manifolds that satisfies
	\[
	\lim_{n\to\infty}\frac{\sigma(Z_n)}{\chi(Z_n)}=\frac{1}{3}.
	\]
\end{theorem}
\vspace{0.2in}

Theorems \ref{Tgeometric}, ~\ref{Treal}, and~\ref{Tcomplex} provide support to Conjecture~\ref{MainC}. It is tantalizing to wonder if it is sensible to sharpen Conjecture~\ref{MainC} to a tighter constraint on the geography of aspherical $4$-manifolds. It also seems promising to pursue an in-depth study of asphericity, the existence of Einstein metrics, and the circle of ideas of LeBrun's Riemannian Bogomolov--Miyaoka--Yau inequality~\cite{LeB95}. More precisely, it would be interesting to determine whether our examples of non-geometric aspherical $4$-manifolds with $\sigma\neq 0$ saturating Inequality~\eqref{E3sigma} admit Einstein metrics. These manifolds are indeed prime candidates for being aspherical $4$-manifolds with non-vanishing signature, which do \emph{not} support Einstein metrics. No such examples seem to be currently available in the literature, and apparently  the first examples of aspherical $4$-manifolds with $\chi>0$ (but with $\sigma=0$!) supporting no Einstein metrics, appeared only very recently in \cite{DiC22}. On the other hand, aspherical $4$-manifolds with $\chi=\sigma=0$ supporting no Einstein metrics are easy to construct thanks to the work of Hitchin \cite{Hit74}, e.g., non-flat $4$-nilmanifolds. Finally, we observe that LeBrun's work on the Riemannian Bogomolov--Miyaoka--Yau inequality implies  our examples (of non-geometric aspherical $4$-manifolds with $\sigma\neq 0$ saturating Inequality~\eqref{E3sigma}) \emph{cannot} simultaneously admit an Einstein metric and a spin$^c$ structures of almost-complex type with non-trivial Seiberg--Witten invariant. 

\subsection*{Acknowledgments}
LFDC thanks Grigori Avramidi, Igor Belegradek, Claude LeBrun, and Stefano Vidussi for interesting discussions around the signature of $4$-manifolds. Part of this work was carried out at the University of Florida: MG thanks this institution for the hospitality and support. We would like to thank an anonymous referee for their comments, especially for catching a computational mistake in the proof of Proposition~\ref{Pequality2}. This work is partially supported by the NSF grant DMS-2104662.

\section{Signature of extended graph $4$-manifolds with real-hyperbolic pieces}\label{real}	
	
In this section, we provide a combinatorial proof of Theorem~\ref{Treal}. This is a warm-up for the proof of Theorem~\ref{Tgeometric}, but is somewhat more concrete.

%\MG{Aggiunto qualcosa qui e nella frase precedente.}
Recall that Theorem~\ref{Treal} asserts that extended graph $4$-manifolds with real hyperbolic pieces has vanishing signature. As mentioned in the introduction, this class of $4$-manifolds is a variation of the one considered by Lafont, Frigerio, and Sisto. The differences are two: here we do not consider product pieces (but the statement would still be true for those as well), but we allow the hyperbolic pieces of the $4$-manifold to have arbitrary flat $3$-manifolds as cusps, rather than just $3$-tori.

We start by recalling some facts about flat $3$-manifolds.

%\MG{Ho estratto questa roba dalla dimostrazione, perch\'e mi sembra pi\`u prerequisiti che dimostrazione.}
If $\Gamma\backslash \textbf{H}^{4}_{\IR}$ is orientable, then the cusp cross-sections are flat orientable closed $3$-manifolds. Up to (not necessarily orientation-preserving) diffeomorphism, there are six orientable flat compact  $3$-manifolds, denoted in the original classification of Hantzsche and Wendt by the letters $A$ to $F$~\cite{HW35}.
We refer to the interesting paper of Ratcliffe and Tschantz~\cite{RT00} for the necessary background on real-hyperbolic $4$-manifolds and their cusps. $A$ is the $3$-torus, and $F$ is the celebrated Hantzsche--Wendt $3$-manifold: this is the unique flat homology $3$-sphere.

We also refer to Scott's survey~\cite{Scott} for the classification of Seifert fibrations on flat $3$-manifolds that we recall now. First, let us set the notation: we denote with $M(g; e; r_1, \dots, r_n)$ the $3$-manifold that Seifert fibres over an orientable surface of genus $g$ (if $g \ge 0$) or over a non-orientable surface with first $\IZ/2\IZ$-Betti number $-g$ if $g < 0$, with Euler number $e$, and $n \ge 0$ singular fibres parametrized by $r_1, \dots, r_n \in \IQ$.

\begin{enumerate}[label = \Alph*:, itemsep = 3pt]
\item This is $T^3 = M(1; 0; )$.
\item This is the unit cotangent bundle of the Klein bottle; it has two Seifert fibrations: $M(0;-2;\frac12,\frac12,\frac12,\frac12) = M(-2;0;)$ (for instance, see~\cite[Section~10.14]{FomenkoMatveev}).
\item This is $M(0;-1;\frac13,\frac13,\frac13)$.
\item This is $M(0;-1;\frac12,\frac14,\frac14)$.
\item This is the $0$-surgery along a trefoil, which is also $M(0;-1;\frac12,\frac13,\frac16)$.
\item This is $M(-1;-1;\frac12,\frac12)$.
\end{enumerate}

The computation of the $\eta$-invariants of Seifert fibred 3-manifolds is due to Ouyang~\cite{Ouyang}, and has been made explicit for flat $3$-manifolds in~\cite{LR00}. Recall that reversing the orientation flips the sign of the $\eta$-invariant. Up to sign, or up to choosing an orientation, the $\eta$-invariants of the six flat $3$-manifolds are:
\[
\eta(A)=\eta(B)=\eta(F)=0, \quad \eta(C)=-\frac{2}{3}, \quad \eta(D)=-1, \quad \eta(E)=-\frac{4}{3}.
\]
Notice that if the $\eta$-invariant does not vanish, then the corresponding $3$-manifold does not admit an orientation-reversing diffeomorphism. 

With these preliminaries into place, let us now prove Theorem~\ref{Treal}.

\begin{proof}[Proof of Theorem~\ref{Treal}] 
Let $Y=\Gamma\backslash \textbf{H}^{4}_{\IR}$ be an orientable, finite-volume hyperbolic $4$-manifold with cusps. If by a slight abuse of notation we denote still by $Y$ the truncated manifold with boundary, then because of Hitchin's formula~\cite{Hit97} (see also~\cite{LR00}):
\[
\boxed{\sigma(Y)=\sum_i\pm\eta(C_i)+\sum_j\pm\eta(D_j)+\sum_k\pm\eta(E_k)},
\]
where the indices $i$, $j$, $k$ count respectively the number of $C$, $D$, $E$ cusps of $Y$. In other words, the signature of the compact, oriented manifold with boundary $Y$ is the sum of the eta-invariants of its boundary components of type $C$, $D$, and $E$. The plus and minus signs depend on which orientation is induced on such boundaries. Thus, say 
\[
Z=\bigcup^M_{\ell=1} Y_\ell
\]
is an orientable $4$-manifold obtained by gluing such orientable truncated real-hyperbolic $4$-manifolds. First, notice that $A, B, C, D, E, F$ have different diffeomorphism types so each of them can only be paired with a cusp of the same diffeomorphism type in the gluing process. Next, because of Novikov's additivity formula we have:
\begin{align}\label{HitchinGraph}
\sigma(Z)=\sum^M_{l=\ell} \sigma(Y_\ell)=\sum^M_{l=1}\Big(\sum^{c_\ell}_{i=1}\pm\eta(C_{i\ell})+\sum^{d_\ell}_{j=1}\pm\eta(D_{j\ell})+\sum^{e_\ell}_{k=1}\pm\eta(E_{k\ell})\Big)
\end{align}
where $c_\ell, d_\ell, e_\ell$ are respectively the numbers of $C, D, E$ cusps of the truncated real-hyperbolic manifold $Y_{\ell}$. Now, since the flat $3$-manifolds $C, D, E$ do not admit orientation-reversing diffeomorphism, we have that any such boundary in say $Y_\ell$ has to be paired with a boundary in some other $Y_{\ell^{\prime}}$ with the \emph{opposite} orientation and as such with the opposite $\eta$-invariant. Thus, the sum in Equation~\eqref{HitchinGraph} is necessarily zero. This concludes the proof.
\end{proof}

\section{Geometric decomposition, signatures, and Euler characteristics}

%\MG{Aggiunto una referenza e qualche frase per chiarire/dare contesto.}
There are $19$ geometries in dimension $4$~\cite{Filipkiewicz}, in the sense of Klein's Erlangen program. That is to say, there are (up to rescaling and isometries) $19$ complete, simply-connected Riemannian $4$-manifolds whose isometry group acts transitively. A $4$-manifold whose universal cover is one of these $19$ manifolds is called \emph{geometric}. Hillman studied the analogue of JSJ decompositions of $4$-manifolds, namely decompositions of $4$-manifolds along hypersurfaces into geometric pieces. Contrarily to the case of $3$-manifolds, not all $4$-manifolds decompose into geometric pieces. Among other things, he proved the following theorem.

\begin{theorem}[Hillman,~{\cite[Theorem~7.2]{Hil02}}]\label{THillman}
If a closed, oriented $4$-manifold $M$ is geometrically decomposed by a (possibly empty or disconnected) closed orientable $3$-manifold $S \subset M$, then either
\begin{enumerate}
\item $M$ is geometric, or
\item $M$ is the total space of an orbifold bundle with general fibre $S^2$ over a hyperbolic $2$-orbifold, or
\item the components of $M \setminus S$ all have geometry $\IH^2\times\IH^2$, or
\item the components of $M \setminus S$ all have geometry $\IH^4$, $\IH^3\times \IE^1$, $\IH^2\times \IE^2$, or $\widetilde{\IS\IL}\times \IE^1$, or
\item the components of $M\setminus S$ have geometry $\IH^2(\IC)$ or $\IF^4$.
\end{enumerate}
In cases (3)--(5), $\chi(M) \ge 0$. In cases (4) and (5), $M$ is aspherical.
\end{theorem}

\begin{remark}
We want to understand when the 4-manifolds in cases (1)--(3) can be aspherical. In case (1), $M$ is aspherical if and only if its universal model is. This rules out the geometries $\IS^4$, $\IS^3\times \IE^1$, $\IS^2\times \IE^2$, $\IS^2\times \IH^2$, $\IS^2\times \IS^2$, $\IC\IP^2$. In case (2), $M$ is never aspherical: in fact, Hillman~\cite[Section~7]{Hil02} shows that the universal cover of $M$ retracts onto $S^2$ (the preimage of a fiber), and in particular $M$ is not aspherical. In case (3), we may or may not obtain aspherical $4$-manifolds. It is easy to construct aspherical examples if each piece is a Hilbert modular surface. In this case, the cross sections of the cusps are infra-solvmanifolds and $\pi_1$-injective (see, for example,~\cite{McR08}). The resulting geometric $4$-manifold is aspherical with $\chi>0$ and $\sigma=0$, see Proposition~\ref{Pchiandsigma}.
\end{remark}

We now compute the Euler characteristic and signature of a geometrically decomposable $4$-manifolds as curvature integrals. More precisely, we express these characteristic numbers as a sum of curvature integrals over each non-compact geometric piece. Each geometry has its complete locally homogeneous metric, but, for ease of notation, we do not explicitly refer to each metric during the proof of the following result.

\begin{proposition}\label{Pchiandsigma}
Suppose that $M$ is a closed oriented $4$-manifold that is geometrically decomposed along $S$ into $N$ geometric pieces $M_1, \dots, M_N$ as above. Let $\mu_\ell$, $W_\ell^\pm$, $\mathring{\Ric}_\ell$, and $s_\ell$ denote the Riemannian volume form on $M_\ell$, the $\pm$-parts of the Weyl tensor, the trace-free Ricci tensor, and the scalar curvature of $M_\ell$ (with its standard metric and the induced orientation), respectively. Then
\begin{align*}
\chi(M) &= \frac1{8\pi^2} \sum_{\ell=1}^N \int_{M_\ell} \big|W^+_\ell\big|^2 + \big|W^-_\ell\big|^2 - \frac12\big|\mathring{\Ric}_\ell\big|^2 + \frac{s_\ell^2}{24} d\mu_\ell \\
\sigma(M) &= \frac1{12\pi^2} \sum_{\ell=1}^N \int_{M_\ell} \big|W^+_\ell\big|^2 - \big|W^-_\ell\big|^2d\mu_\ell.
\end{align*}
\end{proposition}

\begin{proof}
Since $M$ is oriented, $M_\ell$ has an induced orientation from $M$. The $\pm$-parts of the Weyl tensor get exchanged when we reverse orientation, so we need to be somewhat careful with the orientation on the model geometry we are using when computing the characteristic numbers above. In other words, each piece $M_\ell$ has a standard homogeneous metric but may or may not have the standard orientation. The first equality is a direct consequence of the Gauss--Bonnet theorem for complete finite volume $4$-manifolds (see for example \cite{CG85}) and the additivity of the Euler characteristic for even-dimensional manifolds (since $\chi(S) = 0$). 

The second identity is similar. The Atiyah--Patodi--Singer index theorem for cuspidal metrics \cite[Theorem 3.6]{DW07} tells us that for each geometric piece $M_\ell$ we have:
\[
\sigma(M_\ell) = \frac1{12\pi^2} \int_{M_\ell} \big|W^+_\ell\big|^2 - \big|W^-_\ell\big|^2d\mu_\ell - \frac12 \eta(\partial M_\ell),
\]
where $\eta$ denotes the limiting $\eta$-invariant of the boundary. Recall that $\eta$ is additive with respect to disjoint unions and changes sign when reversing orientation. Now observe that each component of $S$ is the boundary of exactly two components of $M_i$ and $M_j$ for some (possibly coinciding) $1\le i,j \le N$. However, the orientations induced on $S$ by $M_i$ and $M_j$ are opposite to each other, so their contributions to the $\eta$-invariant cancel out. That is, summing $\sigma(M_\ell)$ over all $\ell$, the summands with the $\eta$-invariants cancel out, leaving us with the desired expression.
\end{proof}

We now state three lemmas containing all the curvature computations required to apply Proposition~\ref{Pchiandsigma} in order to prove Theorem~\ref{Tgeometric}. To shorten up the notation, we write:
\begin{align*}
\Delta_-(M_\ell) := \chi^{(2)}(M_\ell) - 3\sigma^{(2)}(M_\ell) &= \frac1{8\pi^2}\int_{M_\ell} 3\big|W^-_\ell\big|^2 - \big|W^+_\ell\big|^2 - \frac12\big|\mathring{\Ric}_\ell\big|^2 + \frac{s_\ell^2}{24} d\mu_\ell,\\
\Delta_+(M_\ell) := \chi^{(2)}(M_\ell) + 3\sigma^{(2)}(M_\ell) &= \frac1{8\pi^2}\int_{M_\ell} 3\big|W^+_\ell\big|^2 - \big|W^-_\ell\big|^2 - \frac12\big|\mathring{\Ric}_\ell\big|^2 + \frac{s_\ell^2}{24} d\mu_\ell.
\end{align*}
%\MG{new sentence for new notation}
Here $\chi^{(2)}$ and $\sigma^{(2)}$ denote the $L^2$-Euler characteristic and $L^2$-signature, respectively.

\begin{lemma}\label{LnoWeyl}
The geometries $\IH^2\times\IH^2$, $\IH^4$, $\IH^3\times \IE^1$, $\IH^2\times \IE^2$, and $\widetilde{\IS\IL}\times \IE^1$ have all $|W^+| = |W^-|$.
\end{lemma}

\begin{proof}[Sketch of proof]
Each of these geometries admits an orientation-reversing isometry, as both $\IH^n$ and $\IE^n$ do, and each of them has a hyperbolic or a flat direction. Since reversing the orientation of a 4-manifold $M$ swaps the two bundles $\Lambda^\pm M$ but does not change the norms of the Weyl operators, we immediately obtain that $|W^+| = |W^-|$.
\end{proof}

\begin{lemma}[Hirzebruch proportionality~\cite{Mum77}]\label{LH2C}
If $M_\ell$ is of type $\IH^2(\IC^2)$, then $\Delta_-(M_\ell) = 0$.
\end{lemma}

%\MG{Come ho scoperto da poco, $\IF_4$ \`e la metrica sul fibrato tangente di $\IH^2$. Carino, no?}
\begin{lemma}\label{LF4}
The metric on $\IF^4$ satisfies:
\[
\big|W^\pm_{\IF^4}\big|^2 = \frac38, \qquad \big|\mathring{\Ric}_{\IF^4}\big|^2 = \frac94, \qquad s_{\IF^4} = -3.
\]
\end{lemma}

\begin{proof}[Proof of Lemma~\ref{LF4}]
The metric on $\IF^4$ is given, in coordinates $(x,y,p,q)$ with $y > 0$, by\footnote{We note here that there are two typos in~\cite[Page~8]{Bernt}. The first is a sign $-$ before $dy^2$. The second is more subtle, and is the summand $2y\,dp\,dq$ instead of $2x\,dp\,dq$. The mistake is in the passage from complex coordinates to real coordinates.}
\[
g_{\IF^4} = \frac{dx^2 + dy^2}{y^2} + \frac{(x^2+y^2)dp^2 + 2x\,dp\,dq + dq^2}y.
\]
We used the Sage~\cite{sagemath} package SageManifolds~\cite{SageManifolds} for the following computations: the code and some comments can be found in the Appendix~\ref{a:sage}.

The scalar curvature and the trace-free Ricci tensor are computed directly by SageManifolds, and one obtains that the scalar curvature is $-3$ and the norm of the trace-free Ricci tensor is $\frac94$.

The computation of the self-dual and anti-self-dual Weyl tensors is a bit more laborious.
The eigenvalues of the self-dual and anti-self-dual Weyl tensors, viewed as operators $\Lambda^\pm M \to \Lambda^\pm M$, which are $\pm\frac12$, $\mp\frac14$, and $\mp\frac14$, so that their squared norms are $\big({\pm\frac12}\big)^2 + \big({\mp\frac14}\big)^2 + \big({\mp\frac14}\big)^2  = \frac38$, as claimed.
\end{proof}

We are now in position to prove the inequality in Theorem~\ref{Tgeometric}.

\begin{proof}[Proof of Equation~\eqref{E3sigma}]
When $M$ is geometric, this is a result of Wall and Kotschik~\cite{Wal86, Kot92}.
%\MG{A quanto pare, \`e un risultato di Wall, e Kotschik ha solo corretto un errore di Wall. Confermi?}
So let us assume that $M$ is not geometric, i.e., we assume $S \neq \varnothing$. Thanks to Proposition~\ref{Pchiandsigma}, it suffices to prove that $\Delta_\pm(M_\ell) \ge 0$ holds for each geometric piece $M_\ell$ in the decomposition.

When the piece has a metric of type $\IH^2\times\IH^2$, $\IH^4$, $\IH^3\times \IE^1$, $\IH^2\times \IE^2$, or $\widetilde{\IS\IL}\times \IE^1$, Lemma~\ref{LnoWeyl} implies that the $L^2$-signature of $M_\ell$ vanishes. When the metric is of type $\IF^4$, by Lemma~\ref{LF4}, the $L^2$-signature of $M_\ell$ also vanishes.

By~\cite[Section~7]{Hil02}, the Euler characteristic is non-negative for all pieces of type $\IH^2\times\IH^2$, $\IH^4$, $\IH^3\times \IE^1$, $\IH^2\times \IE^2$, $\widetilde{\IS\IL}\times \IE^1$, or $\IF^4$, so the inequalities $\Delta_\pm(M_\ell) \ge 0$ are verified as $\sigma^{(2)}(M_\ell) = 0$ and $\chi^{(2)}(M_\ell)=\chi(M_\ell) \ge 0$. In the remaining case, that of $\IH^2(\IC)$, Hirzebruch's proportionality (Lemma~\ref{LH2C}) tells us that $\Delta_-(M_\ell) = 0$. Since $\chi^{(2)}(M_\ell) > 0$, we have $\sigma^{(2)}(M_\ell) > 0$ and so $\Delta_+(M_\ell) > 0$. To summarize, for every geometric piece $M_\ell$ in the geometric decomposition, we have $\Delta_\pm(M_\ell) \ge 0$. So, by Proposition~\ref{Pchiandsigma}, the inequality $\chi(M) \ge |{3\sigma(M)}|$ holds.
\end{proof}

\begin{remark}
That pieces of type $\IF^4$ do not contribute to the signature and Euler characteristic of the total manifold also follows from the fact that they have a finite cover that is a flat $T^2$-bundle over an open hyperbolic surface. The metric on a $T^2$ can be shrunk through flat metrics to have arbitrarily small volume, so the contribution of the integrals given by the $\IF^4$ pieces can be made arbitrarily small, and in particular it has to vanish (by integrality of $\sigma$ and $\chi$, for instance).
\end{remark}

%\MG{ri-aggiunto il remark, che ci serve dopo.}
\begin{remark}\label{rmk:chi=3sigma}
From the proof of the Inequality~\eqref{E3sigma}, we see that equality can only be attained if either $\chi = 0$ or $\chi , \sigma > 0$ (respectively, $\sigma < 0)$, all pieces of the geometric decomposition are of type $\IF^4$ or $\IH^2(\IC)$, and either all pieces of the latter type have the \emph{complex} orientation or they all have the \emph{anti-complex} orientation. This will be used in Section~\ref{Schi=3sigma}.
\end{remark}

\section{Non-geometric $4$-manifolds with $\chi = 3\sigma>0$}\label{Schi=3sigma}
%\LFDC{I added $\sigma>0$}
The main goal of this section is to exhibit infinitely many examples of non-complex, closed, aspherical 4-manifolds $M$ with $\chi(M) = 3\sigma(M)>0$. Our construction relies on the $\IH^2(\IC)$ and $\IF^4$ geometries, and it produces examples of geometrically decomposable aspherical $4$-manifolds as in class (5) of Theorem \ref{THillman}.

In~\cite[Chapter~13.3]{Hil02}, Hillman produces an example $X_0$ of a $4$-manifold with geometry $\IF^4$, whose double is a torus bundle over a genus-$2$ surface. The $4$-manifold $X_0$ is a $T^2$-bundle over a once-punctured torus $T_*$. Call $\alpha$ and $\beta$ the generators of\footnote{We omit base-points to make the notation lighter, but this will cause no harm.} $\pi_1(T_*)$, whose corresponding homology classes are a positive basis of $H_1(T_*)$. On $X_0$, the monodromy of the $T^2$-bundle along $\alpha$ is $X = {\tiny \left(\begin{array}{@{}c@{\ }c@{}}2 & 1\\1 & 1\end{array}\right)} \in SL_2(\IZ)$, and along $\beta$ it is $Y = {\tiny \left(\begin{array}{@{}c@{\ }c@{}}1 & 1\\1 & 2\end{array}\right)} \in SL_2(\IZ)$. The monodromy around the puncture (with the orientation induced by the interior of $T^*$, and \emph{not} by the neighborhood of the puncture) is given by the commutator $[X,Y^{-1}] = {\tiny \left(\begin{array}{@{}c@{\ }c@{}}-1 & -6\\0 & -1\end{array}\right)}$. That is to say, the boundary of the cusp of $X_0$ is a circle bundle over a Klein bottle.

Consider the quadruple (non-normal) cover $p\colon S_{**} \to T_*$ from a twice-punctured genus-$2$ surface $S_{**}$ to $T_*$. This cover is obtained by composing two double covers as follows (see Figure~\ref{f:cover}):

\begin{itemize}
\item[(i)] a cover $T_{**} \to T_*$ from a twice-punctured torus $T_{**}$ to $T_*$ induced by the restriction of a double cover of $T$ (the one-point compactification of $T_*$);
\item[(ii)] a cover $S_{**} \to T_{**}$ which extends to ramified cover $S \to T$ with branching locus equal to the two punctures. More precisely, the two loops around the punctures represent the same non-trivial class in $H_1(T_{**}; \IF_2)$, so there is a double cover of $T_{**}$ that non-trivially covers each of them.
\end{itemize}

\begin{figure}[h]
\centering
\includegraphics[scale=0.6]{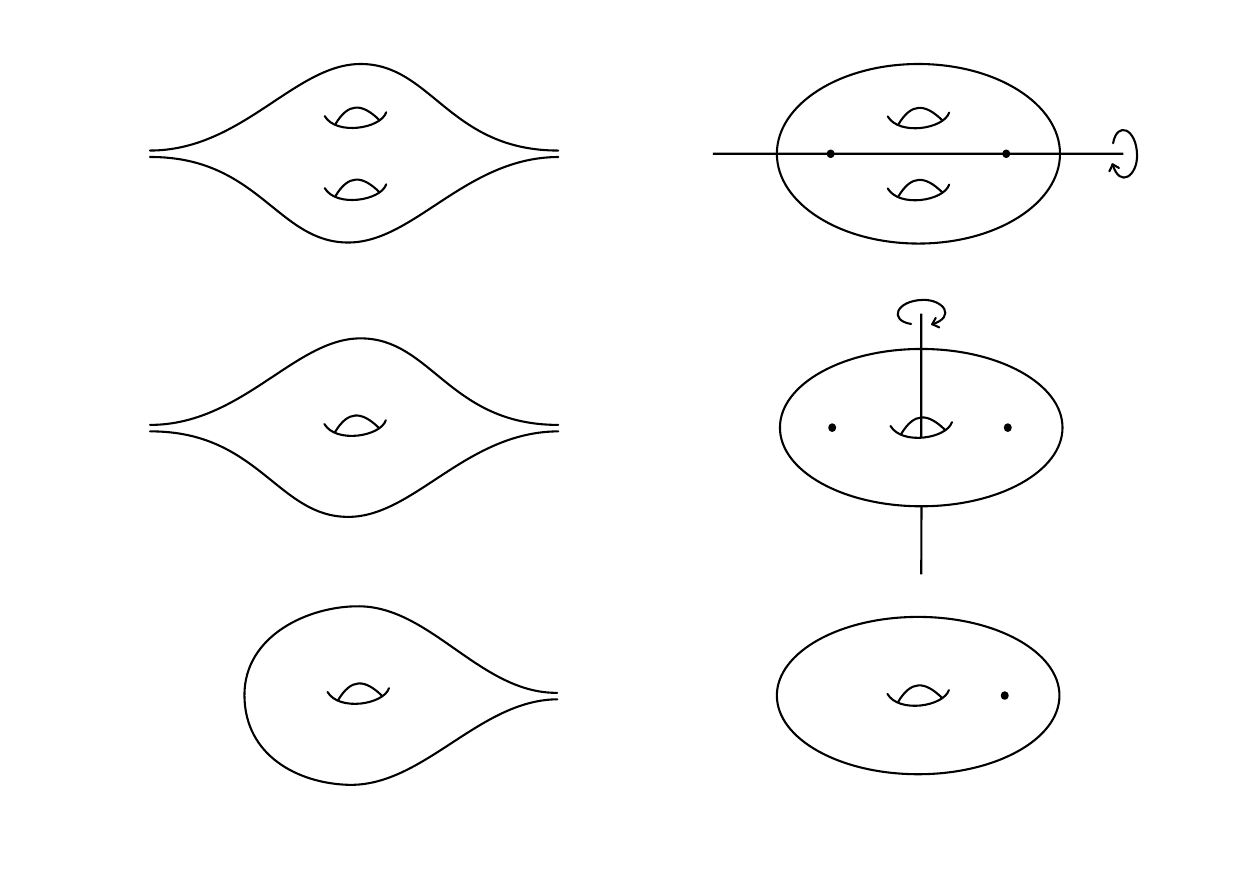}
\caption{A quadruple cover $S_{**} \to T_*$. On the left, from top to bottom $S_{**}$, $T_{**}$, and $T_*$. On the right their compactifications, $S$, $T$, and $T$.}\label{f:cover}
\end{figure}

We now pick any cover $p$ fitting the description above.

\begin{remark}
No cover $p$ as above can be normal. If it were, it would be Abelian, since groups of order $4$ are Abelian. Abelian covers correspond to maps of the fundamental group of the base that factor through the Abelianization. Since the loop around the puncture in $T_*$ is null-homologous, it is trivially covered in each Abelian cover.
\end{remark}

Call $X \to S_{**}$ the pull-back of the bundle $X_0 \to T_*$ to $S_{**}$ via $p$. Notice that the first cover creates two cusps with the same ideal boundary. The second cover non-trivially double covers the cusps, by squaring the monodromy around them. That is, the boundaries of the two cusps of $X$ are both $T^2$-bundles over $S^1$ with monodromy ${\tiny \left(\begin{array}{@{}c@{\ }c@{}}-1 & -6\\0 & -1\end{array}\right)}^2 = {\tiny \left(\begin{array}{@{}c@{\ }c@{}}1 & 12\\0 & 1\end{array}\right)}$. In particular, they are both circle bundles over $T^2$ with Euler number $12$.
%\MG{Questo dovrebbe essere un $-12$ in alto a destra... Non sono capace di moltiplicare matrici.}

In~\cite[Section 4]{DG23}, for each $n$ we constructed a complex-hyperbolic $4$-manifold $H_n$ with $\chi(H_{n})=n$ whose ideal boundary consists of four circle bundles over $T^2$ with Euler number $n$. This is the case with respect to the orientation coming from the complex structure. Next, we select $H_{12}$ and we fix a positive integer $m$. We can glue $m$ copies of $H_{12}$ to $2m$ copies of $X$ along their boundary to obtain a closed and \emph{connected} $4$-manifold $M_m$. We pair boundary components with opposite orientation in order to get a closed, oriented $4$-manifold $M_m$. In other words, we equip the complex hyperbolic pieces with the standard orientation, and we reverse the orientation of the pieces coming from copies of $X$. Note that the diffeomorphism class of $M_m$ depends on the particular choice of gluings. 

\begin{proposition}\label{Pequality}
For each $m$, $\chi(M_m) = 3\sigma(M_m) = 12m$. Moreover, for each $m$, $M_m$ does not support a complex structure and is not geometric.
\end{proposition}

\begin{proof}
We have $\chi(M_m) = 3\sigma(M_m)$ because:
\begin{align*}
\chi(M_m) &= m\chi^{(2)}(H_{12}) + 2m\chi^{(2)}(X) = m\chi^{(2)}(H_{12}) = 12m,\\
3\sigma(M_m) &= m\cdot 3\sigma^{(2)}(H_{12}) + 2m\cdot 3\sigma^{(2)}(X) = 3m\sigma^{(2)}(H_{12}),
\end{align*}
and, by Hirzebruch's proportionality, $\sigma^{(2)}(H_{12})=4$.

Next we show that $M_m$ is neither geometric nor complex. Choose a cusp of a copy of $H_{12}$ in $M_m$, and call $B$ the corresponding $3$-manifold in $M_m$. $B$ is $\pi_1$-injective in $M_m$ (see~\cite[Chapter~7]{Hil02}) and $\pi_1(B)$ contains a copy of $\IZ^2$, since it is a $T^2$-bundle over a circle.
If $M_m$ were complex or geometric, then it would be a ball quotient:
\begin{itemize}
\item if $M_m$ is complex, it saturates the Bogomolov--Miyaoka--Yau inequality, so it is a ball quotient by \cite{Yau77};
\item if $M_m$ is geometric, it either has the geometry of type $\IP^2(\IC)$ or of type $\IH^2(\IC)$, by~\cite[Theorem~6.1]{Wal86}. Since it is aspherical, it has to be the latter.
\end{itemize}

However, fundamental group of compact ball quotients do not contain non-cyclic Abelian subgroups by Preissmann~\cite{Preissmann} (see, for example,~\cite[Chapter~12]{doC92}), so $M_m$ cannot be a ball quotient, and in particular it is neither geometric nor complex.
\end{proof}

There are plenty more examples one can construct, for instance by taking higher covers of $S_{**}$, and gluing complex-hyperbolic 4-manifolds with cusps whose Euler numbers are multiples of $12$. \\

We conclude by sketching another construction based on the second example in~\cite[Chapter~13.3]{Hil02}. Hillman constructs a 4-manifold $X_0'$ of type $\IF^4$ that is a $T^2$-bundle over a pair-of-pants $P$. The monodromy maps around the boundary components are represented by the matrices (in Hillman's notation)
\[
U = {\tiny \left(\begin{array}{@{}c@{\ }c@{}}1 & 0\\2 & 1\end{array}\right)}, \quad
V = {\tiny \left(\begin{array}{@{}c@{\ }c@{}}1 & 2\\0 & 1\end{array}\right)},\quad  
VU^{-1} = {\tiny \left(\begin{array}{@{}c@{\ }c@{}}-3 & 2\\-2 & 1\end{array}\right)} \sim {\tiny \left(\begin{array}{@{}c@{\ }c@{}}-1 & 2\\0 & -1\end{array}\right)}.
\]
Similarly as above we find a degree-4 non-normal cover $P_* \to P$, where $P_*$ is a six-punctured sphere, of the pair-of-pants that squares all monodromies around the punctures.
%\MG{Ho scritto un po' pi\`u esplicitamente com'\`e fatto il rivestimento $P_* \to P$.}
This cover comes from two double covers of the $2$-sphere onto itself, each branched over two points: first we branch over the first two boundary components of $P$, and then over the two preimages of the other component. The net effect is that we are squaring the monodromy around each boundary component, and doubling each boundary component.
Pulling back the torus bundle to $P_*$, we get a $T^2$-bundle $X'$ with a geometry of type $\IF^4$ with six punctures, where the monodromies are given by:
\[
{\tiny \left(\begin{array}{@{}c@{\ }c@{}}1 & 4\\0 & 1\end{array}\right)} \mbox{ four times, and }{\tiny \left(\begin{array}{@{}c@{\ }c@{}}1 & -4\\0 & 1\end{array}\right)} \mbox{ twice.}
\]

%\MG{Avevamo cambiato orientazione due volte, senza motivo... adesso $X'$ ha l'orientazione giusta}
We can now take $2m$ copies of $-X'_0$ and glue half of the torus bundles with Euler number $-4$ to the boundaries of $m$ copies of the complex-hyperbolic $4$-manifold $H_4$ constructed in~\cite{DG23}. We are left with a $4$-manifold whose ends are $4m$ torus bundles, half of which have Euler number $-4$ and half $+4$. We can glue these boundaries together to obtain a closed, connected, geometrically decomposed, oriented $4$-manifold that we call $M'_m$. The same argument as above shows the following proposition.

\begin{proposition}\label{Pequality2}
For each $m$, $\chi(M'_m) = 3\sigma(M'_m) = 4m$. Moreover, for each $m$, $M'_m$ does not support a complex structure and is not geometric.
\end{proposition}

\section{Signature of higher graph $4$-manifolds with complex-hyperbolic pieces}\label{complex}

In this section, we prove Theorem \ref{Tcomplex}. Finally, we prove that the inequality in Theorem \ref{Tcomplex} is sharp and, in doing so, we also construct the first examples of higher graph $4$-manifolds with non-vanishing signature.

Let $Y$ be a complex-hyperbolic surface with nilmanifold cusps. It is well-known that $Y$ admit a smooth toroidal compactification $(X, D)$ of log-general type~\cite[Proposition 4.7]{DiC12}. It is also well-known that $X$ is a projective smooth surface, and $D$ is a reduced divisor consisting of pairwise disjoint smooth elliptic curves with negative self-intersection. These elliptic divisors are in one-to-one correspondence with the cusps of $Y$. We can now explicitly compute the $\eta$-invariant of a torus-like cusp by applying Hirzebruch's proportionality principle in the non-compact setting.
By Hirzebruch's proportionality~\cite{Mum77}, we know that
\[
3\bar{c}_2(X, D)=\bar{c}^2_1(X, D),
\]
where the $\bar{c}_i$'s are the log-Chern numbers of the pair $(X, D)$. It is well-known (see for example \cite[Section 4]{DiC12}) that
\[
\bar{c}_2(X, D)=\chi(X\setminus D)=\chi(Y)>0,
\]
and
\[
\bar{c}^2_1(X, D)=c^2_1(X)-D^2=K^2_{X}-D^2=2\chi(X)+3\sigma(X)-D^2.
\]
Let $D=D_1\cup \dots \cup D_b$, so that $Y$ has exactly $b>0$ cusps. Moreover, say that
\[
D^2_i=-n_i,\quad i=1, \dots, b,
\]
where for each $i$, $n_i$ is a strictly positive integer. Since by Mayer--Vietoris $\chi(X)=\chi(Y)$, we obtain that
\[
\chi(Y)=3\sigma(X)+n_1+\dots+n_b,
\]
so that
\[
\sigma(X)=\frac{\chi(Y)}{3}+\frac{D^2}{3}.
\]
With a slight abuse of notation, let $Y$ also denote the truncated complex-hyperbolic surface with cusps. We have that the signature of this $4$-manifold is given by
\begin{align}\label{bsignature}
\boxed{\sigma(Y)=\sigma(X)+b=\frac{\chi(Y)}{3}-\sum^b_{i=1}\Big(\frac{n_i}{3}-1\Big)}.
\end{align}
This formula computes the $\eta$-invariant of a torus-like cusp in terms of the self-intersection of the associated compactifying divisor. Moreover, it implies the following. 

\begin{lemma}\label{orientation}
	Let $Y$ be a truncated complex-hyperbolic surface with infra-nilmanifold cusps equipped with the orientation coming from the complex orientation. Then each boundary component is oriented in such a way that is finitely covered by a circle bundle over a torus with positive Euler number.
\end{lemma}

\begin{proof}
	It is well-known that any complex-hyperbolic surface is finitely covered by a complex-hyperbolic surface with nilmanifold cusps, see for example Section 2 in~\cite{DiC12}. As discussed at the beginning of Section~\ref{complex}, such surfaces have a well defined compactification by tori of negative self-intersection.
	In particular,~\eqref{bsignature} implies that the tori in this compactification have negative self-intersection, so their tubular neighborhoods are disc bundles over $T^2$ with \emph{negative} Euler number, and so their boundaries are circle bundles over a torus with negative Euler number.
	When we view these boundary components as part of the boundary $Y$, which has an orientation coming from the complex structure, we change their orientation and so they are circle bundles over tori with \emph{positive} Euler number.
\end{proof}

Before proving that Inequality~\eqref{TI} is sharp, we briefly study higher graph $4$-manifolds whose vertices are complex-hyperbolic surfaces that admit smooth toroidal compactifications birational to Abelian and bielliptic surfaces. As pioneered by Hirzebruch~\cite{Hir84}, and later extended by Holzapfel~\cite{Hol80}, the first author and Stover~\cite{DS17, DS19}, and the authors~\cite{DG23}, there are a lot of explicit constructions of smooth toroidal compactifications of complex-hyperbolic surfaces that are birational to Abelian or bielliptic surfaces.  The goal of this discussion is to show that the signature higher graph $4$-manifolds produced by gluing these of these particular complex-hyperbolic surfaces is always zero. Some explicit examples of such higher graph $4$-manifolds with zero signature appeared in~\cite[Section 6.4]{DH21}.

\begin{proposition}\label{holzapfel}
	Let $Z$ be a closed, oriented, higher graph $4$-manifolds whose vertices are truncated complex-hyperbolic surfaces with torus-like cusps whose toroidal compactifications are birational to Abelian or bielliptic surfaces. We then have
	\[
	\sigma(Z)=0.
	\]
\end{proposition}
\begin{proof}
	If $(X, D)$ is a smooth toroidal compactification birational to an Abelian or bielliptic surface, we then have
	\[
	\chi(X)=-\sigma(X).
	\]
	Thus, in this particular case we obtain
	\[
	\sigma(X)=\frac{D^2}{4}, \quad \sigma(Y)=-\sum^{b}_{i=1}\Big(\frac{n_{i}}{4}-1\Big),
	\]
	where $Y$ is the complex-hyperbolic surface with $b\geq 1$ cusps that we are compactifying and $-n_1, \dots -n_b$ are the self-intersections of the compactifying divisors of $Y$.
	
	Suppose now that $Z$ is assembled by gluing together $N$ pieces $Y_1,\dots, Y_N$ with the complex orientation and $M$ pieces $-Y_{N+1}, \dots, -Y_{N+M}$ (here the $-$ sign denotes orientation-reversal: that is to say, for all $j = 1,\dots, M+N$, $Y_j$ is a complex-hyperbolic surface, with the complex orientation).
	Call $b_i$ the number of cusps of $Y_i$, and let $(-n_{ij})_j$ be the self-intersections of the compactifying divisors of $Y_i$, so that the corresponding boundary components of $Y_i$ have Euler number $(n_{ij})_j$.
	
	Since circle bundles over tori are determined (up to diffeomorphism) by their Euler number, the only way we can glue boundary components is by matching a component with Euler number $n_{ij}$ in $Y_i$ (for some $i \le N$ with a component with Euler number $n_{i'j'} = n_{ij}$ in $-Y_{i'}$ (for some $i' > N$).
	
	We then obtain, in the same notation as above,
	\[
	\sigma(Z)=-\sum^N_{i=1}\sum^{b_i}_{j=1}\Big(\frac{n_{ij}}{4}-1\Big)+\sum^{M+N}_{i=N+1}\sum^{b_i}_{j=1}\Big(\frac{n_{ij}}{4}-1\Big)=0.\qedhere
	\]
\end{proof}

\vspace{0.2in}
For a closed, oriented, $4$-manifold we have
\[
\chi \equiv \sigma \pmod 2.
\]
Thus, a higher graph $4$-manifolds manufactured by assembling complex-hyperbolic surfaces with Abelian and bielliptic compactifications always has even Euler number. Since there are examples of such ball quotient compactifications with Euler number one~\cite{Hir84, DS17}, we can easily use these atomic components to construct closed, oriented, higher graph $4$-manifolds of any positive even Euler number and zero signature.

%\MG{Tolto la domanda, messo questo paragrafo. Ho anche leggermente modificato quello successivo per coerenza.}
This circle of ideas lead us to look for closed, oriented, higher graph $4$-manifold $Z$ whose vertices are truncated complex-hyperbolic surfaces with torus-like cusp and for which $\sigma(Z)\neq 0$. To construct such examples, we use a construction of Hirzebruch~\cite{Hir84}. More precisely, recall that Hirzebruch constructs, among other things, a sequence of smooth toroidal compactifications $X_n$, $n=2, 3, \dots$ that are minimal and of general type. Moreover, he computes that
\[
\chi(X_n)=n^7, \quad \sigma(X_n)=\frac{n^5(n^2-4)}{3}.
\]
The number of cusps of the associated complex-hyperbolic surface with cusp $Y_n$ is given by
\[
\ell_n=4 n^4.
\]
Each of the associated elliptic curve has self-intersection $-n$. Thus, we can take $X_2$ that has $\chi(X_2)=128$, and $\ell_2=64$, and fill its cups with 16 copies of the $H_2$ with the reverse orientation. Here, for $n\geq 2$, we denote by $H_{n}$ the complex-hyperbolic surface with four torus-like cusps with Euler number $n$ and such that $\chi(H_{n})=n$ constructed in Section 4.2 of \cite{DG23}. We then denote by $\bar{H}_n$ the reversed-oriented surface. By applying the formula in \eqref{bsignature}, we obtain that
\[
\sigma(Z_2)=\sigma\Big(X_2\cup \bigcup^{16}_{k=1}\bar{H}_2\Big) = \frac{1}{3}\big(128-32\big)=32.
\]
More generally, since $\ell_n$ is always divisible by 4, we can more generally construct $Z_n$ for any $n$. More precisely, we have
\[
\sigma(Z_n)=\sigma\Big(X_n\cup \bigcup^{n^4}_{k=1}\bar{H}_n\Big) = \frac{1}{3}\big(n^7-n^5), \quad \chi(Z_n)=n^7+n^5.
\]
Notice that as in Hirzebruch's examples, one has
\[
\boxed{\lim_{n\to\infty}\frac{\sigma(Z_n)}{\chi(Z_n)}=\frac{1}{3}}.
\]

We can summarize this discussion into a proposition.

\begin{proposition}\label{sharpness}
	For each integer $n\geq 2$, there exists a closed, oriented, higher graph $4$-manifold $Z_n$ such that
	\[
	\chi(Z_{n})=n^7+n^5, \quad \sigma(Z_{n})=\frac{1}{3}\big(n^7-n^5).
	\]
\end{proposition}
	
We can now provide all of the details for the proof of Theorem~\ref{Tcomplex}.
	
\begin{proof}[Proof of Theorem~\ref{Tcomplex}]
%\MG{Non ho capito perch\'e mi hai cassato questa dimostrazione. Quella che era scritta prima era incompleta.}
As mentioned in Remark~\ref{rmk:chi=3sigma}, equality in Equation~\eqref{E3sigma} can only be attained if all complex-hyperbolic pieces have the complex-induced orientation (or all have the anti-complex orientation). However, under the assumptions of Theorem~\ref{TI} and thanks to Lemma~\ref{orientation}, with the complex orientation all the complex pieces have boundary components that are finitely covered by circle bundles over the torus with positive Euler number.

%^\MG{Ho dovuto aggiungere dei pezzi: secondo me la dimostrazione vecchia era solo per nil e non per infranil}
In particular, these are Seifert fibered $3$-manifolds~\cite[Section~4]{Scott} that admit a unique Seifert fibration~\cite[Theorem~3.8]{Scott}. This Seifert fibration has positive Euler number, since it is finally covered by an $S^1$-bundle over $T^2$ with positive Euler number, and the sign of the Euler number is preserved under covers~\cite[Theorem~3.6]{Scott}. In particular, no two of these $3$-manifolds are orientation-reversing diffeomorphic, so that every non-geometric higher graph 4-manifold has to have pieces with both orientations, and therefore the inequality has to be strict. (More precisely, two consecutive vertices of a higher graph 4-manifold have to have complex and anti-complex orientation.)

The sharpness of such inequality is then a consequence of Proposition~\ref{sharpness}. Similarly, the examples of higher graph $4$-manifolds with non-zero signature come from the explicit family of closed manifolds constructed in Proposition~\ref{sharpness}.
\end{proof}

\section{New examples of ball quotient compactifications of general type}

In this section, we look at a class of examples stemming from our investigation. More precisely, we refine a classical construction of Hirzebruch to give a new infinite family of complex-hyperbolic surfaces with toroidal compactifications that are minimal and of general type.

More precisely, recall that Hirzebruch~\cite{Hir84} constructed a family $\{X_n\}_{n > 1}$ of complex surfaces of general type that asymptotically approach the Bogomolov--Miyaoka--Yau line, i.e., $c_1^2(X_n)/c_2(X_n)$ tends to $3$ as $n$ goes to infinity. Each $X_n$ is constructed as an Abelian $n^3$-fold branched cover of a surface $Y_n$, where $Y_n$ is the blow-up of an Abelian surface $E \times E$ at its $n$-torsion points. Calling $C_n$ the cyclic group of order $n$, we can view this construction as the quotient of an action of $C_n^4/C_n \cong C_n^3$ on $X_n$ (here $C_n$ sits diagonally inside $C_n^4$).

When $n = 2m$ is even, we consider the index-$8$ subgroup $G = C_m^3 < C_n^3$ and the corresponding action of $G$ on $X_n$. The quotient, that we call $Z_n$, is an $8$-fold cover of $Y_n$ with the same branching set.

\begin{proposition}
$Z_n$ is a minimal surface of general type. In $Z_n$, the preimages of the $(-1)$-curves in $Y_n$ are non-singular curves of genus $1$ which saturate the logarithmic Bogomolov--Miyaoka--Yau inequality. In particular, their complement is a ball quotient.
\end{proposition}

\begin{proof}
We compute the characteristic numbers of $Z_n$. We begin with $\chi(Z_n) = c_2(Z_n)$, the topological Euler characteristic/second Chern class. Recall that $Y_n$ is the blow-up of an Abelian surface $A$ at $n^4$ points, so $\chi(Y_n) = n^4$. Since $Z_n$ is an $8$-fold cover of $Y_n$, ramified over tori, then $\chi(Z_n) = 8\chi(Y_n) = 8n^4$. We now want to compute $c_1(Z_n) = -K_{Z_n}$. In order to do that, we look a bit closer at $Y_n$ and at the projection $Z_n \to Y_n$. The ramification locus inside $Y_n$ is the union of $4n^2$ tori of self-intersection $-n^2$. Call $B$ the reduced divisor in comprising all these tori $Y_n$, and $R$ its (reduced) preimage in $Z_n$. Since $Z_n \to Y_n$ is of degree 8 and is ramified of degree $2$ at each point of $R$, we have that each component of $R$ has self-intersection $-n^2/2$, and that each component of $B$ lifts to $4$ components of $R$.

The canonical class $K_{Y_n}$ is represented by the divisor $\sum E_i$, where $E_1, \dots, E_{n^4}$ are the $n^4$ exceptional divisors of the blow-ups $Y_n \to A$.
By construction, each component of $E_i$ intersects four components of $B$ transversely at distinct points.
We call $L$ the preimage of the union of all the $E_i$'s in $Z_n$. We claim that each component $E_i$ of $E$ lifts to a component $L_i$ of $L$, and that each of these is a torus of self-intersection $-8$.

%First, we show that the preimage of $E_i$ is connected. By the construction of~\cite[Section 1]{Hir84}, the cover $X_n \to Y_n$ restricts to the cover of $E_i$ associated to the subgroup $\ker(H_1(E_i \setminus B) \to C_n^4/C_n)$, where $E_i \setminus B$ is a four-punctured sphere, so $H_1(E_i \setminus B) = \IZ^4/\IZ$ (generated by the four meridians of the components of $B$ meeting $E_i$) and the homomorphism sends each meridian of $B \cap E_i$ to a distinct coordinate-generator of $C^4_n$. Looking at the sub-cover $Z_n \to Y_n$, we see that the cover  is induced by the corresponding homomorphism $\ker(H_1(E_i \setminus B) \to C_2^4/C_2)$, which has index $8$. Since the cover $Z_n \to Y_n$ has degree $8$, the preimage of $E_i$ is connected.

%\MG{The old paragraph is commented out, now I've rearranged it and made it a bit better.}
First, we show that the preimage of $E_i$ is connected. $E_i\setminus B$ is a four-punctured sphere, so $H_1(E_i \setminus B) = \IZ^4/\IZ$, generated by the four meridians of the components of $B$ meeting $E_i$.
Recall from~\cite{Hir84} that the cover $Z_n \to Y_n$ is associated to a morphism $H_1(Y_n \setminus B) \to C_n^4/C_n$ that is surjective when restricted to the subspace of $H_1(Y_n \setminus B)$.
Since the inclusion $E_i \setminus B \hookrightarrow Y_n\setminus B$ induces a map on homology whose image is exactly the subspace generated by the meridians of $B$, the restriction of the cover $Z_n \setminus R \to Y_n \setminus B$ to $E_i\setminus B$ is onto. It follows that the induced cover is connected, so the preimage of $E_i$ in $Z_n$ is connected.

By the Riemann--Hurwitz formula, the Euler characteristic of $L_i$ is $0$, so it is indeed a torus. Since $E_i$ has self-intersection $-1$ and $Z_n \to Y_n$ has degree $8$ and does not ramify over $E_i$, then $L_i\cdot L_i = -8$. Moreover, each intersection point of $B$ with $E$ lifts to four points of intersections of $R$ with $L$ (since the ramification index at each point of $R$ is $2$ and the degree of the cover is $8$), so that $R\cdot L = 16n^4$.

By the ramification formula, since $Z_n \to Y_n$ has index $2$ on each component of its ramification divisor $R$, we have that $K_{Z_n} = L+R$. Since both $L$ and $R$ are genus-1 divisors of negative self-intersection, $K_{Z_n}$ is free of exceptional divisors and therefore $Z_n$ is minimal. We now have
\[
c_1(Z_n)^2 = K_{Z_n}^2 = L^2 + 2R\cdot L + R^2 = -8\cdot n^4 + 2\cdot 16n^4 + -\frac{n^2}2 \cdot 16n^2  = 16n^4.
\]
Since $c_1(Z_n)^2$ is positive for each $n$, $Z_n$ is of general type. We have also shown that $L$ contains $n^4$ tori of self-intersection $-8$, so that
\[
L^2 = -8n^4 = 3\cdot 8n^4 - 16n^4 = 3c_2(Z_n) - c_1(Z_n)^2.
\]
It follows that $(Z_n, L)$ saturates the logarithmic Bogomolov--Miyaoka--Yau inequality \cite{TY86}, and $Z_n \setminus L$ is a ball quotient.
\end{proof}

\begin{remark}
When $n=2$, $Z_2 = X_2$, and the previous proposition shows that $X_2$ is the toroidal compactification of two distinct ball quotients. One of them is the complement of the ramification divisor (which consists of $64$ tori of self-intersection $-2$), and the other is the complement of the preimages of the exceptional divisor (which consists of sixteen tori of self-intersection $-8$), as in the statement of the proposition above. The two ball quotients are distinct, since they have a different number of cusps. An interesting thing to point about $X_2$ is that it is \emph{minimal} and \emph{of general type}. In~\cite{DS16}, the first author and Stover constructed \emph{non-minimal} and \emph{non-general} type surfaces that are simultaneously the toroidal compactifications of multiple ball quotients. They are blow-ups of Abelian surfaces.
\end{remark}

\appendix

\section{Sage code}\label{a:sage}

Here we provide the Sage~\cite{sagemath} code we used for our computations, using the SageManifolds package~\cite{SageManifolds}.

First, we define the manifold $M$ and the metric in the chart $X$, with coordinates $(p,q,x,y) \in \IR^3 \times (0,\infty)$.
We also define the point $O = (0,0,0,1)$ in these coordinates.

\vskip 0,2 cm

\texttt{M = Manifold(4, \textquotesingle F4\textquotesingle, structure=\textquotesingle Riemannian\textquotesingle)}

\texttt{X.<p,q,x,y> = M.chart(r"p q x y:(0,+oo)")}

\texttt{O = M.point((0,0,0,1), chart=X, name=\textquotesingle O\textquotesingle)}

\vskip 0,2 cm

We now define and print the metric $g$ in the chart.

\vskip 0,2 cm

\texttt{g = M.metric()}

\texttt{g[0,0] = (x{\^{}}2+y{\^{}}2)/y}

\texttt{g[0,1] = x/y}

\texttt{g[1,1] = 1/y}

\texttt{g[2,2] = 1/y{\^{}}2}

\texttt{g[3,3] = 1/y{\^{}}2}

\texttt{g[:]}

\vskip 0,2cm

The Ricci curvature $s_g$ is computed and printed by the following lines:
\vskip 0,2cm

\texttt{s = g.ricci\_scalar()}

\texttt{s.display()}

\vskip 0,2cm

Since the metric is homogeneous, it now suffices to compute the Ricci and Weyl tensors at $O$.

The trace-free part of the Ricci tensor at $O$ is computed as follows.

\vskip 0,2cm

\texttt{R = g.riemann()}

\texttt{Rdown = R.down(g)}

\texttt{Ricop = Rdown[\textquotesingle\_ijkl\textquotesingle]*g.up(g)[\textquotesingle\^{}ik\textquotesingle]}

\texttt{Rico = Ricop - (1/4)*s*g}

\texttt{Rico.at(O).display()}

\vskip 0,2cm

The first line computes the Riemann tensor.
The second lowers all indices, and the third contracts the second and third index with the metric.
The fourth line removes the appropriate multiple of the metric.
The last line displays the trace-free Ricci tensor at $O$, which is $\frac34(dp^2 + dq^2) - \frac34 (dx^2 + dy^2)$.
The squared norm of the trace-free Ricci tensor at $O$ is therefore $4\big({-\frac34}\big)^2 = \frac94$.

Finally, we compute the Weyl operator $W$, as a tensor of type $(2,2)$, in the chart, and evaluate it at the point $O$.

\vskip 0,2 cm

\texttt{C = g.weyl()}

\texttt{Wop = C[\textquotesingle\_ijk\^{}l\textquotesingle]*g.up(g)[\textquotesingle\^{}im\textquotesingle]}

\texttt{Wop.at(O).display()}

\vskip 0,2cm

The output of the computation can be rearranged as follows:

\begin{align*}
+ \frac12 \partial_p \otimes \partial_q \otimes dx \otimes dy - \frac12 \partial_p \otimes \partial_q \otimes dy \otimes dx - \frac12 \partial_q \otimes \partial_p \otimes dx \otimes dy + \frac12 \partial_q \otimes \partial_p \otimes dy \otimes dx\\
+ \frac14 \partial_p \otimes \partial_x \otimes dq \otimes dy - \frac14 \partial_p \otimes \partial_x \otimes dy \otimes dq - \frac14 \partial_x \otimes \partial_p \otimes dq \otimes dy + \frac14 \partial_x \otimes \partial_p \otimes dy \otimes dq\\
- \frac14 \partial_p \otimes \partial_y \otimes dq \otimes dx + \frac14 \partial_p \otimes \partial_y \otimes dx \otimes dq + \frac14 \partial_y \otimes \partial_p \otimes dq \otimes dx - \frac14 \partial_y \otimes \partial_p \otimes dx \otimes dq\\
- \frac14 \partial_q \otimes \partial_x \otimes dp \otimes dy + \frac14 \partial_q \otimes \partial_x \otimes dy \otimes dp + \frac14 \partial_x \otimes \partial_q \otimes dp \otimes dy - \frac14 \partial_x \otimes \partial_q \otimes dy \otimes dp\\
+ \frac14 \partial_q \otimes \partial_y \otimes dp \otimes dx - \frac14 \partial_q \otimes \partial_y \otimes dx \otimes dp - \frac14 \partial_y \otimes \partial_q \otimes dp \otimes dx + \frac14 \partial_y \otimes \partial_q \otimes dx \otimes dp\\
+ \frac12 \partial_x \otimes \partial_y \otimes dp \otimes dq - \frac12 \partial_x \otimes \partial_y \otimes dq \otimes dp - \frac12 \partial_y \otimes \partial_x \otimes dp \otimes dq + \frac12 \partial_y \otimes \partial_x \otimes dq \otimes dp,
\end{align*}
where we write $\partial_p$ instead of $\frac{\partial}{\partial p}$ (and similarly for the other coordinates).

Each line corresponds to an entry in the matrix for the Weyl operator, but now viewed as an operator in $W^\wedge \in \End(\Lambda^2 M)$, with respect to the basis
\[
\partial_p\wedge \partial_q, \quad \partial_p\wedge \partial_x, \quad \partial_p\wedge \partial_y, \quad \partial_q\wedge \partial_x, \quad \partial_q\wedge \partial_y, \quad \partial_x\wedge \partial_y.
\]
That is to say $W^\wedge$ acts as follows:

\[
\begin{array}{lll}
\partial_p \wedge \partial_q \mapsto \frac12 \partial_x \wedge \partial_y, \quad &\partial_p \wedge \partial_x\mapsto \frac14 \partial_q\wedge \partial_y, \quad &\partial_p \wedge \partial_y \mapsto -\frac14 \partial_q \wedge \partial_x\\
\phantom{a}\\
\partial_q \wedge\partial_x \mapsto -\frac14 \partial_p \wedge \partial_y, &\partial_q \wedge \partial_y \mapsto \frac14 \partial_p \wedge \partial_x, &\partial_x\wedge\partial_y \mapsto \frac12 \partial_p \wedge \partial_q
\end{array}
\]

Since the metric $g$ at $O$, in the base $\partial_p, \partial_q, \partial_x, \partial_y$ is represented by the identity matrix, the Hodge star operator, which is an involution on $\Lambda^2 M$ sends $\partial_p \wedge \partial_q$ to $\partial_x \wedge \partial_y$, $\partial_p \wedge \partial_x$ to $-\partial_q \wedge \partial_y$, and $\partial_p \wedge \partial_y$ to $\partial_q \wedge \partial_x$.
An eigenbasis for $\Lambda^{\pm}$ is given by
\[
\partial_p \wedge \partial_q \pm \partial_x \wedge \partial_y, \quad
\partial_p \wedge \partial_x \mp \partial_q \wedge \partial_y, \quad
\partial_p \wedge \partial_y \pm \partial_q \wedge \partial_x,
\]
and these are also eigenvectors of $W^\wedge$ of eigenvalues $\pm\frac12, \mp\frac14, \mp\frac14$, respectively.
%%%%%%%%%%%%%%%%%%%%

\bibliographystyle{amsalpha}
\bibliography{sigma_vs_chi}

\end{document}